\documentclass[12pt,a4paper]{amsart}
\usepackage{amsfonts,amsthm,amsmath,amscd,amssymb,xcolor}
\usepackage[margin=2cm]{geometry}

\usepackage{upref}
\usepackage{bbm}
\usepackage{abraces}
\usepackage{dutchcal}
\usepackage{mymacros}
\usepackage[colorlinks=true, linkcolor=blue]{hyperref}

\newcommand{\pnorm}[1]{\Vert #1 \Vert_{\ell^p}}

\renewcommand{\Re}{\operatorname{Re}}

\title{Quasi-nilpotency of generalized Volterra operators on sequence spaces}
\author{N. CHALMOUKIS}
\author{G. STYLOGIANNIS}

\address{Dipartimento di Matematica, Universit\`a di Bologna, 40126, Bologna, Italy}\email{nikolaos.chalmoukis2@unibo.it}
\address{Department of Mathematics, University of Thessaloniki, 541 24, Thessaloniki, Greece} \email{stylog@math.auth.gr}

\thanks{The first author is a member of INdAM}

\subjclass[2010]{40C05; 47B37; 40G05 ;46A45}

\keywords{Schur multipliers, Hankel operators,  Generalized Volterra operator}

\begin{document}
\maketitle

\begin{abstract}
We study the quasi-nilpotency of generalized Volterra operators on spaces  of power series with Taylor coefficients in weighted $\ell^p$ spaces $1<p<+\infty$ . Our main result is that when an analytic symbol $g$ is a multiplier for a weighted $\ell^p$ space, then the corresponding generalized Volterra operator $T_g$ is bounded on the same space and  quasi-nilpotent, i.e. its spectrum is $\{0\}.$ This improves a previous result of A. Limani and B. Malman in the case of sequence spaces. Also combined with known results about multipliers of $\ell^p$ spaces we give non trivial examples of bounded quasi-nilpotent generalized Volterra operators on $\ell^p$.

We approach the problem by introducing what we call Schur multipliers for lower triangular matrices and we construct a family of Schur multipliers for lower triangular matrices on $\ell^p, 1<p<\infty$ related to summability kernels. To demonstrate the power of our results we also find a new class of Schur multipliers for Hankel operators on $\ell^2 $, extending a result of E. Ricard.
\end{abstract}
\section{Introduction}
Let $H(\bD)$ be the space of holomorphic functions in the unit disc equipped with the topology of local uniform convergence which renders it a Fr\'echet space. A {\it Banach space of analytic functions } in the unit disc is a Banach space $X \subseteq H(\bD)$ such that the inclusion is continuous. One of the most classical examples is the family of Hardy spaces $H^p, 1\leq p < + \infty$ which consists of functions $f\in H(\bD)$ such that 
\[\norm{f}_{H^p}:= \sup_{0<r<1} \Big( \int_0^{2\pi}|f(re^{i\theta})|^pd\theta \Big)^{1/p} < + \infty.\]

For $p=\infty$, $H^\infty$ is just the algebra of bounded analytic functions in the unit disc. The literature on Hardy spaces is vast, but here we would like to concentrate on the Taylor coefficients of Hardy functions. Suppose we consider the Banach space of analytic functions $\ell^p_A$ which consists of functions $f\in H(\bD)$ such that its Taylor coefficients $\hat{f}(n)$ belong to $\ell^p$, i.e.,
\[ \norm{f}_{\ell^p_A} : = \Big( \sum_{n=0}^\infty |\hat{f}(n)|^p \Big)^{1/p} < + \infty.\]

Then the classical Hausdorff-Young inequalities \cite[Theorem 6.1]{Duren} say that for $1\leq p < + \infty $ and $q$ its conjugate exponent $(p^{-1}+q^{-1}=1)$,
\[ H^p \subseteq \ell^q_A, \quad (1\leq p \leq 2), \quad \text{and} \quad \ell^q_A \subseteq H^p, \quad (2\leq p < +\infty). \]

It is also known that the inclusions are strict unless $p=2$ in which case $\ell^2_A = H^2.$ This superficially might suggest a connection between the Hardy space $H^p$ and the sequence space $\ell^q_A$, but although function theory and operator theory of Hardy spaces is a mature and well developed subject, only recently there has been some interest on studying $\ell^p_A, p\neq 2$ as function spaces, and the theory is still at its infancy. The monograph \cite{Cheng2020} of R. Cheng et al contains most of what is currently known. 

In particular a problem which has attracted a lot of interest is the characterization of multipliers for $\ell^p_A$. In general if $X$ is Banach space of analytic functions in the unit disc, its {\it multiplier space} $\Mult(X)$ is the space of all $g\in H(\bD)$ such that $g \cdot f \in X, \forall f \in X.$ In other words (via the closed graph theorem) $g$ is a multiplier of $X$ if the {\it multiplication operator } $M_f(g) = g \cdot f$ is bounded on $X$.

If $X=H^p$ it is known and not very difficult to prove that $\Mult(H^p) = H^\infty$. In general it is true that $\Mult(X) \subseteq H^\infty,$ but it can happen that the inclusion is strict. In particular $\Mult(\ell^p_A)$ is a strict subset of $H^\infty$ if $p\neq 2$  \cite{Lebedev1998} as it doesn't contain the singular inner function 
$ \exp (- \frac{1+z}{1-z}) $. As far as we know there exist no simple analytic characterization of the functions in the space $\Mult(\ell^p_A)$ when $p\neq 2$.

In this article we would like to concentrate on a different operator acting on $\ell^p_A$ and its connection with multipliers. Let $g\in H(\bD)$ and consider the {\it generalized Volterra operator}; 
\[ T_g(f)(z):=\int_0^zf(t)g'(t)dt. \]
There exists also the {\it generalized Ces\'aro operator} 
\[ C_g(f)(z):=\frac{1}{z} \int_0^zf(t)g'(t)dt. \]
The reason for this terminology is that if $g(z)=\log\frac{1}{1-z}$ then the matrix which represents the operator $C_g$ with respect to the  orthonormal basis of the monomials in $\ell^2_A$ is the Ces\'aro matrix 
\[ C:=
\begin{pmatrix}
1       & 0       & 0       &\cdots \\
\frac12 & \frac12 & 0       &\cdots \\
\frac13 &\frac13  & \frac13 & \cdots \\
\vdots  &         &         & \ddots
\end{pmatrix}.
\]
The generalized Volterra operator seems to have been introduced first by Pommerenke \cite{Pommerenke1977}, in connection to the John--Nirenberg inequality, and since then it has been intensively studied in terms of its boundedness and compactness properties in a variety of settings (see \cite{Aleman1995, Brevig2019, Pau2016}), but also finer properties such as Schatten ideals, spectral properties and quasi-nilpotency  have been examined on various spaces (see \cite{Aleman2009, Pelez2012, Limani2020}).

In terms of boundedness, on the spaces that we work with,  the operators $C_g$ and $T_g$ are bounded simultaneously therefore we will use the one that is more convenient.  One can ask, as for multipliers, which is the space of symbols $g$ such that $T_g$ acts boundedly on a given Banach space of analytic functions $X$. There is no commonly accepted notation for this space, but we will denote it by $T(X)$. It can be proven that $T(X)$ carries a Banach space structure \cite{Siskakis1999}.  In this case as well, the state of the art is the same. There is a nice characterization of holomorphic functions $g$ which give rise to a bounded operator $T_g$ on $H^p$. The space in question \cite{Aleman1995}, which does not depend on $p$, turns out to be the well known $BMOA$ space, the space of analytic functions in $H^1$ such that their non tangential boundary values have bounded mean oscillation. Whereas for the sequence spaces $\ell^p_A$ almost nothing is known. Even giving a nontrivial example of a function in $ T(\ell^p_A) $ requires a moments thought. In fact such an example is provided by Hardy's inequality, which asserts that for every sequence of positive numbers $\{a_n\} \in \ell^p$, 

\[ \sum_{n=1}^\infty \Big( \frac{1}{n}\sum_{k=1}^n a_k \Big)^p \leq \big(\frac{p}{p-1}\big)^p \sum_{n=1}^\infty a_n^p. \]

which in our language can be written, 
\[ \norm{C_{\log\frac{1}{1-z}} f }_{\ell^p_A} \leq \frac{p}{p-1} \norm{f}_{\ell^p_A}, \quad \forall f \in \ell^p_A. \]

Our aim here is not to give an analytic characterization of the space $T(\ell^p_A)$,  a probably infinitely complex problem, but we content ourselves with the more modest problem of studying the relation between $\Mult(\ell^p_A) $ and $ T(\ell^p_A) $ and try to understand the spectral picture of $T_g$ when $g\in \Mult(\ell^p_A).$ Our motivation is the following theorem \cite[Theorem 2.2]{Limani2020}

\begin{thm}
     Let $X$ be a Banach space of analytic functions in the unit disc, which contains the constants and such that the algebra $\cB(X)$ of bounded linear operators on $X$ contains the multiplication operators $M_g$ and the generalized Volterra operator whenever $g\in H^\infty.$ Then we have that $\sigma(T_g|X)=\{0\}$ whenever $T_g$ lies in the norm closure of $\{ T_h: h\in H^\infty \}$ in $\cB(X).$
\end{thm}

Notwithstanding the apparent interest of this theorem in some respects is not optimal. First of all it is required that the Banach space has as multiplier algebra the whole $H^\infty$ which as we have seen is a rather special case. Furthermore, and maybe more importantly it is required {\it a priori} the finiteness of the norm of $T_g$ whenever $g$ is a multiplier for $X$. 

Our main result is an optimal version of this theorem when $X$ is a (weighted) $\ell^p$ space. Suppose that $\omega := \{ \omega_n \} $ is a positive weight function such that $\lim_{n} \omega_n/\omega_{n+1} = 1$, then the weighted sequence space $\ell^p_A(\omega)$ is defined as the space of $f\in H(\bD)$ such that 
\[ \norm{f}_{\ell^p_A(\omega)}:=\Big( \sum_{n=0}^\infty |\hat{f}(n)|^p\omega_n \Big)^{1/p}. \]

\begin{thm} \label{main:thm}
     Let $\ell^p_A(\omega)$ be a weighted sequence space. Then the multiplier algebra of $\ell^p_A(\omega)$ is contained in the space of symbols that induce bounded generalized Volterra operators. Explicitly,
     \[ \Mult(\ell^p_A(\omega)) \subseteq T(\ell^p_A(\omega)). \] 
    Furthermore if $g$ is in the norm closure of $\Mult(\ell^p_A(\omega))$  in $T(\ell^p_A(\omega))$ then $T_g$ is quasi--nilpotent, i.e.,  $\sigma(T_g|\ell^p_A(\omega)) = \{0\}.$
\end{thm}

The surprising fact is that one does not have to know anything about the multiplier algebra $\Mult(\ell^p_A(\omega))$ and yet can conclude that the generalized Volterra operator is bounded and it has trivial spectrum. 
It is worth also mentioning that in this generality this theorem is new also in the Hilbert space case $p=2$. We will try to illustrate this with an example.

\begin{exa}
Consider the Dirichlet space $\cD$, the space of analytic functions $f\in \bD$ such that 
\[ \int_{\bD}|f'(z)|^2dA(z) < + \infty, \]
where $dA$ is the normalized Lebesgue area measure in the unit disc. This is a Hilbert space of analytic functions equipped with the norm 
\[ \norm{f}^2_\cD:= \sum_{n=0}^\infty(n+1)|\hat{f}(n)|^2. \]
It is known that the multipliers of the Dirichlet space have a description in terms of {\it Carleson measures} for $\cD$, that is positive Borel measures $\mu$ in the unit disc such that $\cD \subseteq L^2(d\mu, \bD).$ Let $\cX:=\{g\in H(\bD): |g'|^2dA \,\, \text{is a Carleson measure} \}$. Then \cite{Stegenga1980} $\Mult(\cD)=H^\infty \cap \cX$. It is also easy to check that $T(\cD)=\cX$ which of course contains $ \Mult(\cD)$ as it is to be expected by Theorem \ref{main:thm}. Furthermore Theorem \ref{main:thm} tells us that if $g$ is in the closure of $ H^\infty \cap \cX$ in $\cX$,  $T_g$ is quasi--nilpotent in the Dirichlet space.
\end{exa}

The next example shows that combining some known results with our main theorem we can give some non trivial examples of quasi--nilpotent generalized Volterra operators on $\ell^p_A. $

\begin{exa}
Let $r>1$ and $0\leq \alpha < \frac{\pi}{2}$ and consider the region 
\[ \Omega_{r,\alpha}:= \{z : |z|<r \} \setminus \{ z: \arg(z-1)\leq \alpha \}. \]
In \cite{Vinogradov1974}  Vinogradov showed that if a function $g$ is analytic and bounded on $\Omega_{r,\alpha}$ then it is a multiplier for $\ell^p_A.$ For example let $B$ a Blaschke product with all its zeros belonging in $[0,1)$, then by Vinogradov's Theorem and Theorem \ref{main:thm} the operator $T_B$ is bounded and quasi-nilpotent on all $\ell^p_A, 1<p<+\infty.$
\end{exa}


\section{Schur multipliers for lower triangular matrices}

In this section we will discuss the ideas that are behind the proof of Theorem \ref{main:thm}. Temporarily we can forget about analytic functions and work with the Banach spaces $\ell^p$. Since $\ell^p$ is a sequence space it is convenient to work with the matrix representation of operators with respect to the standard {\it unconditional basis}  $e_n := \{ \delta_{kn}\} _{k}  $.  With respect to this basis a bounded linear operator $A\in \cB(\ell^p)$ has a representation 
 \[   [A]= ( \inner{Ae_k}{e_n} )_{0\leq k ,n } \]
as an infinite matrix, where the pairing $\inner{\cdot}{\cdot}$ is the standard $\ell^p-$ pairing. Usually we denote by $a_{kn}$ the entries of the matrix representation. From now on we shall make little distinction between a bounded linear operator itself and its matrix representation.
A {\it Schur multiplier} is an infinite matrix $S=(\sigma_{kn})_{kn}$ such that \[ S\odot A := ( \sigma_{kn} a_{kn} )_{kn} \in \cB(\ell^p), \, \forall A \in \cB(\ell^p). \]
The pointwise product $\odot$ is usually referred to as Schur or Hadamard multiplication. An application of the closed graph theorem shows, that we can naturally define a norm on the space of Schur multipliers 
\[ \norm{S}_{\cS_p}:= \sup \{ \norm{ S \odot A }_{\cB(\ell^p)} : \,\,\, {\norm{ A }_{\cB(\ell^p)} \leq 1 } \}. \]
With this norm and the $\odot$ multiplication the space of all Schur multipliers, denoted by $\cS_p$ is  a unital commutative Banach algebra with identity element the matrix $\mathbbm{1}:= (1)_{kn}$ . The study of this space, as an object with inherent interest, has been initiated in the seminal paper of Bennett \cite{Bennett1977}. Since then, there has been a growing interest in Schur multipliers (see for example \cite{Aleksandrov2002}, \cite{Blasco2019}, \cite{Davidson2007} ).

In particular it seems the case that quite a few problems in operator theory, but also in the study of spaces of analytic functions can be formulated in terms of Schur multipliers. We will return to this point later in connection to our main theorem. 

To see the connection between Schur multipliers and the discussion before let us write $C_g$ in ``matrix form"
\begin{align*}
    [C_g] = 
    \begin{pmatrix}
    \hat{g}(1)&0                  &     0            & \cdots \\
    \hat{g}(2)&\frac12 \hat{g}(1) &     0            & \cdots \\
    \hat{g}(3)&\frac23 \hat{g}(2) &\frac13 \hat{g}(1)& \cdots \\
    \vdots    & \vdots            &                  & \ddots
    \end{pmatrix}
    =
     \begin{pmatrix}
     1        &0                  &     0            & \cdots \\
     1        &\frac12            &     0            & \cdots \\
     1        &\frac23            &\frac13           & \cdots \\
    \vdots    & \vdots            &                  & \ddots
    \end{pmatrix}
    \odot
     \begin{pmatrix}
    \hat{g}(1)&0                  &     0            & \cdots \\
    \hat{g}(2)&        \hat{g}(1) &     0            & \cdots \\
    \hat{g}(3)&        \hat{g}(2) &        \hat{g}(1)& \cdots \\
    \vdots    &   \vdots          &                  & \ddots
    \end{pmatrix} 
    = \cF \odot [M_{S^*g}].
\end{align*}
Where  $S^*$ is the {\it backward shift operator}, $S^*g(z):=(g(z)-g(0))/z$ and we have used the notation $\cF:=(1-\frac{k}{n+1})_{0\leq k \leq n}$  \footnote{When writing $A=(a_{kn})_{0\leq k \leq n }.$ we mean that the elements above the main diagonal are zero. Similarly when writing $A=(a_{kn})_{k\neq n} $ we mean that the diagonal elements are zero.}. We shall call $\cF$ the {\it Fejer matrix} because the $n^{th}$ row is just the positive Fourier coefficients of the $n^{th}$ Fejer kernel.

Therefore, with regards to the first half of Theorem \ref{main:thm}, the fact that $T_g$ is bounded whenever $g\in \Mult(\ell^p_A) $, we would be done if we knew that $ \cF \in \cS_p.$ Unfortunately this is not the case.

\begin{prop}
Let $\mathbbm{1}_{L}= (1)_{k\leq n}$. For any $1<p<\infty$ and any $\lambda \in \bC$ we have that $\lambda \cdot \mathbbm{1}_L+ \cF \not\in \cS_p.$
\end{prop}

\begin{proof}
We shall use the following necessary condition for Schur multipliers which is due to Bennett \cite{Bennett1977} for $\cS_2$ and Coine \cite{Coine2018} for $\cS_p, p\neq 2$. Suppose that $S=(\sigma_{kn})$ is a Schur multiplier and that the iterated limits
\[ \lim_k\lim_n \sigma_{kn}= :\ell_1, \,\,\,\, \lim_n\lim_k \sigma_{kn} =: \ell_2 \]
exist. Then $\ell_1=\ell_2$.

Suppose now that $\lambda\cdot \mathbbm{1}_L + \cF$ belongs to $\cS_p$ for some $p\in(1,\infty)$ and some $\lambda \in \bC$. Since the iterated limits for this matrix are respectively $1+\lambda$ and $ 0 $ we conclude that $\lambda = -1.$ 
In this case the matrix is given by
\[ \cF-\mathbbm{1}_L = \Big( \frac{-k}{n+1} \Big)_{0\leq k \leq n}. \]
To see that this matrix is still not a Schur multiplier, consider the discrete Hilbert transform;
\[ \cH : = \big( \frac{1}{n-k} \big)_{k\neq n}. \]
One has,
\[ \big( \frac{-k}{n+1} \big)_{0\leq k \leq n} \odot \cH = \big(\frac{n-k+1}{n-k } \frac{1}{n+1}\big)_{0\leq k < n } - \big( \frac{1}{n-k} \big)_{0\leq k < n}. \]
On the right we have the sum of a bounded operator (dominated by twice  the Ces\'aro matrix) and an unbounded one (the lower truncation of the discrete Hilbert transform), which proves our point.
\end{proof}

 Therefore, if this approach is to be fruitful, one has to relax the requirement that $\cF$ is a Schur multiplier for bounded operators in $\ell^p$. In fact it would be enough to ask the Schur multiplier property for a subclass of matrices that contain the matrices $[M_{S^*g}],$ whenever they are bounded on $\ell^p.$ 

In the sequel we will try to develop a theory for Schur multipliers for lower triangular matrices. By lower triangular matrices we mean matrices of the form $A=(a_{kn})_{0\leq k \leq n}$, that is, matrices that vanish above the main diagonal. 

\begin{defn}
Let $S=(\sigma_{kn})_{0\leq k \leq n}$ a lower triangular matrix. We say that $S$ is a Schur multiplier for lower triangular matrices if for every $A\in\cB(\ell^p)$ which is lower triangular, $S\odot A \in \cB(\ell^p).$  The set of all such matrices $S$ we will denote it by $\cS^L_p.$
\end{defn}

This choice is motivated by several reasons. For one thing, Schur multipliers for lower triangular operators form a unital commutative Banach algebra, which we shall denote by $\cS_p^L$, with identity element the matrix $\mathbbm{1}_L$ and therefore the tools from the classical theory of Banach algebras are available. Moreover, a little less obvious is that  elements in $\cS_2^L$  are in fact Schur multipliers for Hankel operators on $\ell^2$ (see Section \ref{HankelMult}). 

More importantly we are able to prove that $\cS_p^L$ contains interesting elements that do not belong to $\cS_p$, among them, the Fejer matrix $\cF$. We have this more general theorem.

\begin{thm}\label{main}
     Let $\theta \in C(\bR)$ with support in $[-1,1]$. Denote by $\hat{\theta}$ the Fourier transform of $\theta$ and suppose that 
     \[ |\hat{\theta}(x)| \leq C (1+|x|)^{-a}, \]
     for some $a>1$.
     Then the matrix 
    \[ \Theta : = \big\{ \theta\big(\frac{k}{n+1}\big) \big\}_{0\leq k \leq n} \]
     is a Schur multiplier for lower triangular matrices on $\ell^p, 1<p<\infty$.
\end{thm}
The proof of the above theorem, although requires a long calculation, uses only elementary techniques. In fact if one considers the  kernel $k_n^\theta$ corresponding to the generating function $\theta$
\[k_n^\theta(t) := \sum_{|k| \leq n} \theta\big(\frac{k}{n+1} \big) e^{ik t},  \]
the theorem follows by  elementary manipulations and some simple pointwise estimates of these kernels. 

There is a variety of examples of functions $\theta$ satisfying the hypothesis of the theorem. A comprehensive list for example can be found in \cite[Chapter 2.11]{Weisz2017}, but some particular kernels including the so called Riesz kernels are of special interest.

\begin{cor}\label{examples:Schur}
    Let $ \gamma > 0, p > 1$, we define the matrix
    \[ \cF_{\gamma} := \Big( \big(  1- \frac{k}{n+1} \big)^\gamma  \Big)_{0\leq k \leq n}.  \]
    Then $\cF_{\gamma}$ belongs to $\cS_p^L$. Moreover, there exists a constant $C=C(p)>0$ such that $\norm{\cF_N}_{\cS_p^L} \leq C N^2 $ for all $N=1,2, \dots$
\end{cor}
 The quantitative estimate will be important for our last application of the theorem and it is obtained by a careful examination of the constants involved (see Section \ref{proofofmain}).

\subsection{Extensibility  of Schur multipliers for lower triangular operators.}

In this short section we would like to draw attention to a problem connected to multipliers in $\cS^L_p$ that we think is very interesting. Suppose that a lower triangular matrix $S$ has the following property; there exists a matrix $T\in \cS_p$ such that 
\[ \Pi(T):= T \odot \mathbbm{1}_L = S. \]
If this happens we say that $S$ {\it extends to a Schur multiplier in } $\cS_p$. Of course if $S$ is such it is also a Schur multiplier for lower triangular matrices on $\ell^p$, because for $A\in \cB(\ell^p)$, lower triangular $ T\odot A = S \odot A \in \cB(\ell^p). $
The converse is not at all clear. 

\begin{prob} Is it true that every element in $\cS_p^L, 1<p<\infty$ can be extended to a Schur multiplier in $\cS_p$ ?
\end{prob}

At least when $p=2$, a necessary condition that a Schur multiplier for lower triangular operators has to satisfy in order to be extensible is to be {\it completely bounded} as an operator on the space of bounded lower triangular operators. That is because every Schur multiplier in $\ell^2$ is completely bounded \cite[Theorem 5.1]{Pisier2001}.

In the particular case that $a:=\{a_k \}_{k\geq 0}$ and $p=2$ is a sequence of complex numbers and $T_a$ is the corresponding (lower triangular) Toeplitz matrix $T_\alpha : = \{ \alpha_{n-k} \}_{0\leq k \leq n}$, it can be seen without difficulty that this is indeed the case.

\begin{prop}
For $a$ and $T_a$ as before, $T_a \in \cS_2^L$ if and only if 
\[ \tau(z) : = \sum_{k\geq 0} a_k z^k \] is a Cauchy transform of a finite (complex) Borel meausure on the unit circle. 
\end{prop}

If $\tau$ is  is as above, that is, there exists a finite Borel measure $\mu$ such that $a_k$ are the positive Fourier coefficients of $\mu$, then by \cite[Theorem 8.1]{Bennett1977} $T_\mu:=\{\hat{\mu}(n-k)\}_{kn}$ is in $\cS_2$, and it extends $T_a.$
\begin{proof}
One direction is clear. Suppose now that $T_a \in S^L_2$. Let us denote by $*$ the coefficient wise multiplication of two power series. Then for a bounded holomorphic function $h$ we have
\[ \norm{T_a \odot T_h}_{\ell^2} = \norm{T_{\tau * h}}_{\ell^2} = \norm{\tau * h}_{H^\infty} < +\infty.  \]
For the second inequality we have used \cite[Proposition 10.1]{zhu_operator_2007} In other words $\tau$ is a coefficient self-multiplier for $H^\infty$, and this is equivalent \cite[Theorem 10.1.2]{Jevti2016} to being a Cauchy transform of a finite Borel measure.
\end{proof}

\subsection*{Notation} We will use the letter $C$ to denote a general positive constant that depends on some parameters and might change from appearance to appearance. When we want to stress the dependence of $C$ on some parameters $\alpha, \beta, \gamma, \dots$ we write $C=C(\alpha, \beta,\gamma, \dots).$





\section{Proof of the main theorems}\label{proofofmain}

The proof of the the following elementary lemma can be found scattered around the literature. We provide a short proof of it for the sole purpose of completeness.

\begin{lem}\label{continuous:discrete}
Let $\theta \in  C(\bR) $ with support contained in $[-1,1]$, and 
\[ |\hat{\theta}(x)| \leq (1+|x|)^{-a}, \]
for some $a>1$. Then the kernel 
 \[k_{n}^\theta (t) : = \sum_{|k|\leq n}\theta \big( \frac{k}{n+1} \big) e^{ikt} \] 
 satisfies
\begin{enumerate}
    \item $ \norm{k_n^\theta}_{L^1(\mathbb{T})} \leq  \norm{\hat{\theta}}_{L^1(\bR)}. $
    \item $|k_n^{\theta}(x)| \leq C \min\{n+1, (n+1)^{-(a-1)} |x|^{-a} \}, \quad |x|<\pi,$
\end{enumerate} for some positive constant $C=C(\norm{\theta}_{L^\infty(\bR)}, \norm{\hat{\theta}(x)|x|^\alpha}_{L^\infty(\bR)}, a).$
\end{lem}

\begin{proof} We shall use the following normalization for the Fourier transform
\[ \hat{f}(\xi) = \frac{1}{2\pi} \int_{\bR} f(x) e^{-i\xi x} dx. \] 
For $\lambda > 0$ let also $\delta_\lambda(f)(x) = f(\lambda x)$, the dilation operator. 

Consider now the corresponding ``continuous version" of the kernel $k_n^\theta$; 
\begin{align*}{} K^\theta_T(\xi):  & =\frac{1}{2\pi} \int_{-T}^T\theta\big( \frac{x}{T}\big)e^{-ix\xi}dx = \widehat{\delta_{1/T}(\theta)}(\xi) = T \hat{\theta}(T\xi). \end{align*}

By the Poisson summation formula \cite[Theorem 3.2.8 ]{Grafakos2009} we have that 
\[ k^\theta_{n}(x) = \sum_{k\in \bZ} \delta_{1/(n+1)}(\theta)(k)e^{ikx}= \sum_{k\in \bZ} \widehat{\delta_{1/(n+1)}(\theta)}(x+2k\pi) =  \sum_{k \in \bZ} K^\theta_{n+1} (x+2k\pi). \]
Hence, 
\[ \norm{k^\theta_{n}}_{L^1(\mathbb{T})} \leq \sum_{k\in \bZ} \int_{-\pi}^\pi |  K_{n+1}^\theta(x+2k\pi) | dx = \int_{\bR} |K_{n+1}^\theta(x)|dx = \int_{\bR}|\hat{\theta}(x)|dx < + \infty.  \]
Also, by the definition of $k_n^\theta$, 
\[ |k_n^\theta(x)| \leq (2n+1) \norm{\theta}_{L^\infty(\bR)}. \]
Finally for $n\geq 1$,
\begin{align*} |k_{n-1}^\theta(x)| & \leq n \sum_{k\in\mathbb{Z}} |\hat{\theta}(nx+2kn\pi )| \\
& \leq C n\sum_{k\in\mathbb{Z}}\frac{1}{|xn+2k\pi  n|^{a}} \\
& \leq  C \frac{a}{a-1}\frac{1}{|x|^{a}n^{a-1}},\quad |x| < \pi.
\end{align*}

\end{proof}

Theorem \ref{main} will now follow from the following result. 

\begin{lem}\label{Theorem:quantitative}
	 Let  $\{\varphi_n \}\subset L^1(\mathbb{T}) $ be a family of kernels which satisfy:
	\begin{enumerate}
	\item $\varphi_n$ is a trigonometric polynomial of degree $n$.
	    \item $\norm{\varphi_n}_{L^1} \leq \rho $,
	    \item $|\varphi_n (t)| \leq \rho \min \{ n+1, \frac{1}{(n+1)^at^{a+1}} \}, \quad a>0,  -\pi<t<\pi.$
	\end{enumerate}
	Then the matrix
	\[ \Phi:= \big\{ \hat{\varphi}_n(k)   \big\}_{k,n} \]  is a Schur multiplier for lower triangular matrices on $\ell^p$ for $p>1$.
	Furthermore, 
\[ \norm{\Phi}_{\cS_p^L} \leq C \rho, \,\,\, \text{ as } \rho \rightarrow \infty, \]
where $C=C(\alpha,p).$

\end{lem}

\begin{proof}
Suppose now that $ x=(x_k)\in \ell^p $. By H\"older's inequality and the fact that $ \norm{\varphi_{n}}_{L^1}\leq \rho$ we have that,

\begin{align*}
\pnorm{\Phi \odot A (x)}^p= & \sum_{n=0}^{\infty} \left| \sum_{k=0}^{n}a_{nk} \hat{\varphi}_n(k) x_k \right|^p \\
= & \sum_{n=0}^{\infty} \left|\int_{-\pi}^{\pi}\varphi_{n}(t) \sum_{k=0}^{n}a_{nk}x_ke^{ikt}\frac{dt}{2\pi} \right|^p \\
\leq & \rho^{\frac{p}{q}} \sum_{n=0}^{\infty} \int_{-\pi}^{\pi}\left|\varphi_{n}(t)\right|\cdot \left| \sum_{k=0}^{n}a_{nk}x_ke^{ikt}\right|^p \frac{dt}{2\pi} \\
\leq &2^{p-1} \rho^{\frac{p}{q}} \sum_{n=0}^{\infty} \int_{-\pi}^{\pi}\left|\varphi_{n}(t)\right|\cdot \left| \sum_{k=0}^{n}a_{nk}x_k(e^{ikt}-1)\right|^p \frac{dt}{2\pi} \\
  & \,\,\,\,\,\,\,+  2^{p-1} \rho^{p} \sum_{n=0}^{\infty} \left| \sum_{k=0}^{n}a_{nk}x_k \right|^p \\
\leq & 2^{p-1}\rho^p\sum_{n=0}^{\infty}(n+1) \int_{|t|<\frac{1}{n+1}} \left| \sum_{k=0}^{n}a_{nk}x_k(e^{ikt}-1)\right|^p \frac{dt}{2\pi} \\
 & \,\,\,\,\,\, + 2^{p-1}\rho^p \sum_{n=0}^{\infty} \int_{\frac{1}{n+1}<|t|<\pi} \frac{1}{(n+1)^{a}|t|^{a+1}}\left| \sum_{k=0}^{n}a_{nk}x_k(e^{ikt}-1)\right|^p \frac{dt}{2\pi} \\
 & \,\,\,\,\,\,\, + 2^{p-1}\rho^{p} \norm{A}^p \pnorm{x}^p.
\end{align*}
Lets call the two main terms appearing above  (I) and (II) in order of appearance. In order to estimate these terms, for $ n\in \mathbb{N} $ we define \begin{equation}
S_n(t):= \sum_{\lambda=0}^n \left| \sum_{k=0}^{\lambda}a_{\lambda k}x_k(e^{ikt}-1)\right|^p.
\end{equation} Notice that if we define a sequence $ y_k=x_k(e^{ikt}-1) $ for $ 0\leq k \leq n $, $ y_k=0 $ otherwise we have that
\begin{equation}
S_n(t)=  \sum_{\lambda=0}^n \left| \sum_{k=0}^{\lambda}a_{\lambda k}y_k \right|^p \leq \pnorm{A(y)}^p \leq \norm{A}^p \sum_{k=0}^{n}|x_k|^p|e^{ikt}-1|^p.
\end{equation} (Notice that this is the only place where we use the assumption that $A$ is lower triangular.) With this estimate in hand we go back to estimate (I) and (II). Fix $ M>0 $ and by Abel's summation by parts we have
\begin{align*}
\text{I} & =   \int_{-\pi}^{\pi} \sum_{n=0}^{M}  (n+1)  \chi_{[|t|<\frac{1}{n+1}]}(t)  \left| \sum_{k=0}^{n}a_{nk}x_k(e^{ikt}-1)\right|^p \frac{dt}{2\pi}\\
 &=\int_{-\pi}^{\pi} \sum_{n=0}^{M}(n+1) \chi_{[|t|<\frac{1}{n+1}]}(t) \left(S_{n}(t)-S_{n-1}(t))\right)\frac{dt}{2\pi}\\
 &=\int_{-\pi}^{\pi} \sum_{n=0}^{M-1}\left((n+1) \chi_{[|t|<\frac{1}{n+1}]}(t)-(n+2) \chi_{[|t|<\frac{1}{n+2}]}(t)\right)S_{n}(t)\frac{dt}{2\pi}\\
 & \,\,\,\,\,\,\,\,\,\,\, +\int_{-\pi}^{\pi}(M+1)\chi_{[|t|<\frac{1}{M+1}]}(t)S_{M}(t)\frac{dt}{2\pi}\\
 & \leq  \sum_{n=0}^{M-1} (n+1)\int_{\frac{1}{n+2}<|t|<\frac{1}{n+1}}S_n(t)\frac{dt}{2\pi} + (M+1)\int_{-\frac{1}{M+1}}^{\frac{1}{M+1}}S_M(t)\frac{dt}{2\pi}
 \\ & \leq   \sum_{n=0}^{\infty} (n+1)\int_{\frac{1}{n+2}<|t|<\frac{1}{n+1}}S_n(t)\frac{dt}{2\pi} +  \frac{2^{p}}{\pi} \norm{A}^p\pnorm{x}^p \\
 &=:  \text{I}'  +  \frac{2^{p}}{\pi}\norm{A}^p\pnorm{x}^p.
\end{align*}
But,
\begin{align*}
\text{I}' \leq  & 2^p\norm{A}^p \sum_{n=0}^{\infty} (n+1) \int_{\frac{1}{n+2}}^{\frac{1}{n+1}} \sum_{k=0}^{n}|x_k|^p(kt)^p dt \\
&\leq 2^p\norm{A}^p \sum_{k=0}^{\infty}k^p |x_k|^p\sum_{n=k}^{\infty}\frac{1}{(n+1)^{p+1}} \\
&\leq2^p\norm{A}^p \sum_{k=0}^{\infty}k^p |x_k|^p\int_{k}^{\infty}\frac{1}{x^{p+1}}\,dx\\
&\leq  \frac{2^p}{p}\norm{A}^p \pnorm{x}^p.
\end{align*}

For (II) we use a similar method. Fix $ M>0 $,

\begin{align*}
 \sum_{n=0}^{M} \int_{\frac{1}{n+1}<|t|<\pi} & \frac{1}{(n+1)^{a}|t|^{a+1}}\Big| \sum_{k=0}^{n}a_{nk}x_k(e^{ikt}-1)\Big|^p \frac{dt}{2\pi}\\
   = & \int_{-\pi}^{\pi}  \sum_{n=0}^{M}\frac{1}{(n+1)^{a}} \chi_{[\frac{1}{n+1}<|t|<\pi]}(t) \Big| \sum_{k=0}^{n}a_{nk}x_k(e^{ikt}-1)\Big|^p \frac{dt}{2\pi |t|^{a+1}} \\
 \leq & \sum_{n=0}^{\infty}\left(\frac{1}{(n+1)^{a}}-\frac{1}{(n+2)^{a}}\right)\int_{\frac{1}{n+1}<|t|<\pi}S_n(t)\frac{dt}{2\pi|t|^{a+1}}  \\ & \,\,\,\,\,\,\,\,\,\,\,\,\,\,\,\,\,\,\,\, + \frac{1}{(M+1)^{a}}\int_{\frac{1}{M+1}<|t|<\pi}S_M(t)\frac{dt}{2\pi|t|^{a+1}} \\
 \leq & \sum_{n=0}^{\infty}\left(\frac{1}{(n+1)^{a}}-\frac{1}{(n+2)^{a}}\right)\int_{\frac{1}{n+1}<|t|<\pi}\sum_{k=0}^{n}|x_k|^p|e^{ikt}-1|^p\frac{dt}{2\pi |t|^{a+1}}\norm{A}^p  \\ & \,\,\,\,\,\,\,\,\,\,\,\,\,\,\,\,\,\,\, +\frac{2^{p}}{a\pi}\norm{A}^p \pnorm{x}^p \\  = &
 \sum_{k=0}^{\infty}|x_k|^p \sum_{n=k}^{\infty}\left(\frac{1}{(n+1)^{a}}-\frac{1}{(n+2)^{a}}\right)\int_{\frac{1}{n+1}<|t|<\pi}\frac{|e^{ikt}-1|^p}{2\pi|t|^{a+1}}dt\norm{A}^p \\
 & \,\,\,\,\,\,\,\,\,\,\,\,\,\,\,\,\,\,\,\, +\frac{2^{p}}{\pi a} \norm{A}^p \pnorm{x}^p\\
 =& \text{II}'+\frac{2^{p}}{a\pi} \norm{A}^p \pnorm{x}^p.
\end{align*}
We estimate the integral in $\text{II}'$ as follows:
\begin{align*}
\int_{\frac{1}{n}<|t|<\pi}\frac{|e^{ikt}-1|^p}{|t|^{a+1}}\frac{dt}{2\pi} &\leq 2^{p+1}k^p\int_{\frac{1}{n}}^{\frac{1}{k}}t^{p-a-1}\frac{dt}{2\pi} +2^{p+1}\int_{\frac{1}{k}}^{\pi}\frac{1}{t^{a+1}}\frac{dt}{2\pi} \\
&\leq \frac{2^p}{\pi}(\frac{1}{p-a}+\frac{1}{a}) k^{a}.
\end{align*}
Where the last estimate is obtained by directly computing the integrals and omitting the negative terms in the resulting expression. Here note that we can always assume that $a \leq 1 $, so $p>a.$ Consequently,

\begin{align*}
\text{II}'
&\leq \frac{2^p}{\pi}(\frac{1}{p-a}+\frac{1}{a}) k^{a}\sum_{n=k}^{\infty}\left(\frac{1}{(n+1)^{a}}-\frac{1}{(n+2)^{a}}\right) \leq \frac{2^p}{\pi}(\frac{1}{p-a}+\frac{1}{a}).
\end{align*}
\end{proof}

Corollary \ref{examples:Schur} is in fact a direct consequence of Lemma \ref{continuous:discrete} and Lemma \ref{Theorem:quantitative} combined with the simple estimate

\begin{lem}\cite[Lemma 2]{Brown2004}\label{Brown lem}
Let $\gamma>0$ and 
\[ \phi_{\gamma}(x) : =  \max\{0,(1-|x|) \}^\gamma . \]
Then 
\[ |\hat{\phi}_{\gamma}(x)|  \leq c(\gamma) (1+|x|)^{-\min\{ 1,\gamma \}-1}\]
\end{lem}
In order to obtain the quantitative behaviour of the Schur multiplier norm claimed in Corollary \ref{examples:Schur} we need the  quantitative estimates for the constant $c(\gamma) \leq C \gamma^2, \gamma \in \bN$. This is a quite standard calculus argument, integrating by part twice the Fourier transform of of $\hat{\phi_{\gamma}}$, that we omit. 


We can now proceed with the proof of Theorem \ref{main:thm}.  






\begin{proof}[Proof of Theorem \ref{main:thm}]
Let $g\in \Mult(\ell^p_A(\omega))$. It is easy to check that since $\lim_n\frac{\omega_{n+1}}{\omega_n}=1$ then also  $S^*g \in \Mult( \ell^p_A(\omega)).$ Consider the natural isometry $U$ which maps $ \ell^p_A$ surjectively onto $\ell^p_A(\omega)$, and sents $ z^n $ to $\omega_n^{-1/p} z^n$. Consider also the operators $\hat{C_g}, \hat{M}_{S^*g}$ on $\ell^p_A$, where recall that $S^*$ is the backward shift operator,  which are represented by the matrices
\begin{equation}\label{Schur:relation} [\hat{C_g}] =  \big( \big( 1-\frac{k}{n+1} \big) \frac{\hat{g}(n+1-k)\omega_n^{1/p}}{\omega_k^{1/p}} \big)_{0\leq k \leq n} = \big( 1- \frac{k}{n+1} \big)_{0\leq k \leq n } \odot  [\hat{M}_{S^*g}]. \end{equation}

We claim that $U$ is an intertwining operator for the couples $C_g, \hat{C_g}$ and $M_{S^*g}, \hat{M}_{S^*g}$, in other words that for all polynomials $p$, 
\[ \hat{C_g}U(p) = U C_g (p) \quad \text{and} \quad \hat{M}_{S^*g} U(p) = U M_{S^*g} (p). \]

By linearity it is enough to check it for monomials. The verification is left for the reader. 

With this at hand, we can deduce that $\hat{M}_{S^*g}$ is bounded on $\ell^p_A$ since $ M_{S^*g} $ is bounded on $\ell^p_A(\omega).$ Then, equation \eqref{Schur:relation} and Corollary \ref{examples:Schur} allow us to conclude that $ \hat{C_g} $ is also bounded on $ \ell^p_A $ and again using the intertwining relation, $C_g$ is bounded on $\ell^p_A(\omega)$, equivalently $T_g$ is bounded on $\ell^p_A(\omega)$ i.e. $ \norm{T_g}_{\cB(\ell^p_A(\omega))}<+\infty.$

It remains to prove that the spectrum of $T_g$ is trivial whenever $g$ is in the closure of the space of multipliers in $T(\ell^p_A(\omega)).$ For this we shall follow the proof of \cite[Theorem 2.2]{Limani2020} mutatis mutandis. For a non zero complex number $\lambda$ we set, at least formally,  $R_g(\lambda):=(\id-\lambda^{-1}T_g)^{-1}$. Suppose now that $\lambda \in \bC \setminus\{0\}.$ By assumption we can find a function $h\in\Mult(\ell^p_A(\omega))$ such that $\norm{T_g-T_h}_{\cB(\ell^p_A(\omega))}=\norm{T_{g-h}}_{\cB(\ell^p_A(\omega))}$ is small enough such that $ T_{g-h} - \lambda \id$ is invertible. 
For the constant function $1\in \ell^p_A(\omega)$ we have that 
\[ R_g(\lambda) 1 = e^{\frac{h}{\lambda}} \in \Mult(\ell^p_A(\omega)) \subseteq \ell^p_A(\omega).\]
The fact that the exponential of a multiplier is still a multiplier is a consequence of the fact that $\Mult(\ell^p_A(\omega))$ is a Banach algebra. 
Finally if $f\in \ell^p_A(\omega)$ and $f(0)=0$, by \cite[p. 8]{Limani2020} we have that 
\[ R_g(\lambda)f= M_{e^{h/\lambda}} R_{g-h}(\lambda)M_{e^{-h/\lambda}} f +  M_{e^{h/\lambda}} R_{g-h}(\lambda)T_{e^{-h/\lambda}} f.  \]
The first term represents a bounded operator since $e^{\pm h/\lambda}$ is a multiplier and $\lambda \not\in \sigma(T_{g-h}| \ell^p_A(\omega) )$. The second term is bounded since from the first part of the theorem, ${e^{-h/\lambda}} \in \Mult(\ell^p_A(\omega)) \subseteq T(\ell^p_A(\omega)).$ 



\end{proof}

\section{A miscellaneous result}

\subsection{Schur multipliers for Hankel matrices} \label{HankelMult}
In this last section we would like to point out a connection with a result due to Ricard \cite{Ricard2002}. Recall that a {\it Hankel matrx} associated to a sequence of complex numbers $\alpha = \{ \alpha_n \}_{n\geq 0}$ is the infinite matrix 
\[H_\alpha := \begin{pmatrix}
    \alpha_0 & \alpha_1 & \alpha_2 & \cdots \\
    \alpha_1 & \alpha_2 & \alpha_3 & \ddots \\
    \alpha_2 & \alpha_3
    & \ddots   &        \\ 
    \alpha_3 & \ddots   &          &        \\
    \vdots   &          &          & 
\end{pmatrix}. \]

The following theorem concern Hankel operators which act boundedly on $\ell^2$.

\begin{thm}[\cite{Ricard2002}] The matrix 
\[ E := \Big( \frac{k+1}{k+n+1}\Big)_{k,n} \]
    is a Schur multiplier for bounded Hankel matrices on $\ell^2$, meaning that for for every Hankel matrix $H_\alpha$ which is bounded on $\ell^2$ the matrix $E \odot H_\alpha$ is also bounded on $\ell^2$.
\end{thm}

It worth's mentioning that this result apart from the independent interest that it might have, it answers a question of Davidson and  Paulsen about CAR-valued Foguel-Hankel operators which are similar to a contraction (for more details see \cite{Ricard2002} and \cite{Davidson1997}). The following theorem can be considered as an asymmetric version of Ricard's theorem.

\begin{thm} For $ \Re \lambda >0$, the matrices
\[ E_\lambda :=  \big(  \frac{k+1}{ k + \lambda n+1 }\big)_{k,n}\]
are Schur multipliers for bounded Hankel operators on $\ell^2.$
\end{thm}

\begin{proof} 
We decompose the matrix $E_\lambda$ to its lower and upper triangular parts;
\[ E_\lambda: = \Big( \frac{k+1}{ k + \lambda n+1 } \Big)_{0\leq k \leq n} + \Big( \frac{n+1}{\lambda k + n +1 } \Big)_{0\leq k < n}^{T} : = X + Y^T.\]
It suffices to prove that $ X, Y \in \cS_2^L $ because by  a theorem of Bonami and Bruna \cite{Bonami1999}, if $H_\alpha$ is a  bounded Hankel operator on $\ell^2$ then its lower triangular truncation, which remember we denote it by $\Pi(H_\alpha)$ is also bounded on $\ell^2$. Hence,
\[ E_\lambda \odot H_\alpha = X \odot \Pi( H_\alpha ) + [Y \odot (H_\alpha - \Pi(H_\alpha))^T]^T\in \cB(\ell^2). \]

	Let us denote by $X^{\odot N}$ the $N$th Hadamard power of a matrix (i.e. the matrix with every entry elevated to the power $N$).
  By the quantitative estimate in Corollary \ref{examples:Schur}
  \begin{align*} \lim_{N\to \infty } \norm{\cF^{\odot N}}_{\cS_p^L}^{\frac{1}{N}} & = \lim _{N\to \infty }\norm{\cF_{N}}_{\cS_2^L}^{\frac{1}{N}}   \leq \lim _{N\to \infty }(CN^2)^\frac{1}{N} = 1.
  \end{align*}
 Therefore by the spectral radius formula have 
  \[ \sigma ( \cF | \cS_2^L ) \subseteq \overline \bD. \] 
  
  Hence, if $\Re \lambda > 0, \lambda +1  \not\in\sigma(\cF|\cS^L_2)$, in other words the element $ (\lambda+1) \mathbbm{1}_L - \cF $ is invertible in $\cS^L_2 $.  But the inverse is obtained just by taking the algebraic inverse of the entries in the lower triangular part of the matrix, i.e. 
  \[ ((\lambda+1)\mathbbm{1}_L - \cF )^{-1}=( \frac{n+1}{k+\lambda (n +1)})_{0\leq k \leq n }.\]

Since $\cS^L_2$ is an algebra,
\[ (\mathbbm{1}_L - \cF ) \odot ((\lambda+1)\mathbbm{1}_L - \cF )^{-1}  = \big( \frac{k}{k+\lambda(n+1)} \big)_{0\leq k \leq n} \in \cS^L_2.\]

We further have,
\begin{equation}\label{dif} \frac{k+1}{k+\lambda n +1 } - \frac{k}{k + \lambda n + \lambda } = \frac{k+1}{k+\lambda n + \lambda }\frac{\lambda -1 }{ k + \lambda n + 1 } +\frac{1}{\lambda} \frac{1}{\frac{k}{\lambda}+ n + 1 } . \end{equation}
Which proves that also the matrix 
\[ X= \Big( \frac{k+1}{k+\lambda n +1} \Big)_{k\leq n} \]
is a Schur multiplier for lower triangular matrices, since the first matrix on the right hand side of equation \eqref{dif} is bounded (hence in $\cS_2^L $ by \cite[Proposition 2.1]{Bennett1977}) as a product of a bounded Hilbert type matrix and a matrix in $\cS_2^L$ and the second matrix is a bounded Hilbert type matrix.

By a similar reasoning we can show that the matrix $((\lambda+1)\mathbbm{1}_L - \cF )^{-1}$
   also belongs to $\cS_2^L$ because it differs from $Y$ by a  matrix in $\cB(\ell^2)$.

\end{proof}

 We should also mention that our approach is completely different from the one followed by Ricard, which uses Hardy space theory.

It would be interesting to know if the theorem of Bonami and Bruna remains valid when $p\neq 2$, or for weighted $\ell^2$ spaces, but this appears to be a subtle question. The problem seems to be that the proof of Bonami and Bruna uses the Nehari theorem and also some delicate estimates for the bilinear Hilbert transform due to Lacey and Thiele \cite{Lacey1997}.

\subsection*{Acknowledgments}  The authors would like to thank professor A. Siskakis for stimulating discussions on the problem. In fact his question (private communication) whether $\Mult(\ell^p) \subseteq T(\ell^p)$ or not was the initial motivation for considering the problem.
We would also like to thank the anonymous referee for his useful comments. In particular her/his suggestions have lead to a simpler proof of the quasi-nilpotency claim in the main theorem.
\bibliographystyle{plain}
\bibliography{Bibliography}

\end{document}